\documentclass[10pt]{article}
\date{}
\usepackage{amssymb}
\usepackage{amsmath}
\usepackage{amsthm}
\usepackage{graphicx}
\usepackage{indentfirst}
\allowdisplaybreaks[4]
\newtheorem{theorem}{Theorem}

\newtheorem{corollary}[theorem]{Corollary}

\newtheorem{lemma}[theorem]{Lemma}

\numberwithin{equation}{section}
\numberwithin{theorem}{section}
\newcommand{\keywords}[1]{\par\noindent\textbf{Keywords:} #1}
\newcommand{\subjclass}[2]{
  \par\noindent\textbf{Mathematics Subject Classification} #2}

\begin{document}

\title{A Liouville theorem for the CR Yamabe type equation on Sasakian manifolds }

\author{ Biqiang Zhao\footnote{
		E-mail addresses: 2306394354@pku.edu.cn}}
            
\maketitle
\setlength{\parindent}{2em}

\begin{abstract}
In this paper, we study the CR Yamabe type equation 
\begin{align}
         \Delta_b u+F(u)=0 \nonumber
     \end{align}
on complete noncompact $(2n+1)$-dimensional Sasakian manifolds with nonnegative curvature. Under some assumptions, we prove a rigidity result, that is, the manifold is CR isometric to Heisenberg group $\mathbb{H}^n$. The proofs are based on the Jerison-Lee's differential identity combining with integral estimates.

\end{abstract}

\subjclass [{  32V20 · 35J61 · 35R03 · 53C25 }


\section{Introduction}
\label{1}
     In this paper, we study the Liouville type theorems for solutions of semilinear equation 
     \begin{align}
         \Delta_b u+2n^2F(u)=0
     \end{align}
     on Sasakian manifolds. 
     \par
     The equation (1.1) is closely related to the CR Yamabe equation. Let $\Theta$ be the standard contact form on $\mathbb{H}^n$. We consider the conformal form $\theta=u^{\frac{2}{n}}\Theta$, where $u$ is a smooth positive function on $\mathbb{H}^n$. Then the pseudo-Hermitian scalar curvature associated to $\theta$ is the positive constant $4n(n+1)$ if $u$ satisfies the equation
     \begin{align}
         \Delta_b u+2n^2u^{\frac{Q+2}{Q-2}}=0,
     \end{align}
     where $Q=2n+2$ is the homogeneous dimension of $\mathbb{H}^n$. By the work \cite{JL} of Jerison and Lee, there are nontrivial solutions as follows
     \begin{align}
          u(z,t):=\frac{C}{|t+i|z|^2+z\cdot \mu +\lambda|^n},
      \end{align}
      for some $C>0,\ \lambda\in \mathbb{C},\ \mu\in \mathbb{C}^n$ such that $ \mathrm{Im}(\lambda)>\frac{|\mu|^2}{4}$. In particular, Jerison and Lee \cite{JL} obtained the uniqueness of the solutions in case of finite  energy assumption, i.e., $u\in L^{\frac{2Q}{Q-2}}(\mathbb{H}^n)$. Garofalo and Vassilev \cite{GV} also got a uniqueness result under the assumption of cylindrical symmetry. In \cite{CLMR}, Catino, Li, Monticelli, Roncoroni
      proved the classification of solutions, provided $n=1$ or $n\geq 2$ and a suitable control at infinity holds (see also \cite{FV}). Recently, Prajapat and Varghese \cite{PV} got the classification of solutions of (1.1) using
      the moving plane method. For general Sasakian manifolds, Catino, Monticelli, Roncoroni and Wang \cite{CMRW} generalized the result in \cite{CLMR} to complete noncompact $(2n+1)$-dimensional Sasakian manifolds with nonnegative curvature. 
      \par
       In the subcritical case, i.e.,
       \begin{align}
           \Delta_b u+2n^2u^{q}=0,\quad in\quad \mathbb{H}^n,
       \end{align}
       where $1<q<\frac{Q+2}{Q-2}$, Ma and Ou \cite{MO} proved that the only nonnegative solution is the trivial one.
       \par
        Before presenting our main results, we mention the corresponding equation in Euclidean space,
        \begin{align}
            \Delta u+u^{q}=0,\quad in\quad \mathbb{R}^n.
        \end{align}
        In the seminal paper \cite{GS}, Gidas and Spruck proved that there are no positive solutions to (1.5), if $n\geq 3$ and $1<q<\frac{n+2}{n-2}$. For the critical case $q=\frac{n+2}{n-2}$,  Caffarelli, Gidas and Spruck \cite{CGS} showed that any solution to (1.5) is given by 
        \begin{align*}
          u(x)= \left( \frac{1}{a+b|x-x_0|^2}\right)^{\frac{n-2}{2}},
      \end{align*}
      where $a,b>0,\ ab=\frac{1}{n(n-2)} $ and $x_0\in \mathbb{R}^n$. Under the  finite energy assumption, Catino and Monticelli \cite{CM} proved the classification of solutions to (1.5) with $q=\frac{n+2}{n-2}$ in the Riemannian setting when the Ricci curvature is non-negative. Besides, Li and Zhang \cite{LZ} studied the following equation 
      \begin{align}
          \Delta u+f(u)=0,\quad in\quad \mathbb{R}^n,
      \end{align}
      and proved that: If $f$ is locally bounded in $ (0,\infty)$ and $f(s)s^{-\frac{n+2}{n-2}}$ is non-increasing in $ (0,\infty)$. Let $u$ be a positive classical solution of (1.6) with $n\geq 3$, then either for some $b>0$
      \begin{align*}
          bu=\left (  \frac{\mu}{1+\mu ^2|x-\bar{x}|^2}\right )^{\frac{n-2}{2}}
      \end{align*}
      for $\mu >0,\bar{x}\in \mathbb{R}^n$ or $u\equiv a$ for some $a > 0$ such that $f(a) = 0$. Such Liouville type theorems have also been generalized for
      $p$-Laplacian. In \cite{SZ}, Serrin and Zou extended Gida and Spruck's result to quasilinear elliptic equations. We refer to \cite{MMP,GSXX,HSW,CMR} and references therein for further results.
      \par
      Returning to the equation (1.1), we first give some definitions. We say that $f\in C^0([0,\infty))\cap C^1((0,\infty))$ is subcritical with exponent $\sigma$ if 
        \begin{align*}
            f(t)>0,\quad \sigma f(t)-tf^{'}(t)\geq 0,\quad   \forall t>0.
        \end{align*}
        Different from \cite{SZ,Zbq}, we request that $f(t)>0$. It is easy to verify that $f$ is a subcritical function with exponent $\sigma$ implying that $ t^{-\sigma}f(t)$ is non-increasing in $(0,\infty)$. In the following, we denote by $B_R$ the ball of radius $R$ centered at a fixed point with respect to the Carath$\acute{\mathrm{e}}$odory distance (see Section 2) and denote $A_R=B_{2R}\setminus B_{R}$. The volume of $B_R$ is denoted by $V(R)$ .
        \par
        In this paper, we study the equation (1.1) on complete noncompact Sasakian manifolds with nonnegative pseudo-Hermitian Ricci curvature $Ric_b$. Indeed, we have
      \begin{theorem}
          Let $(M^3,HM,J,\theta)$ be a 3-dimensional complete noncompact Sasakian manifold with nonnegative Tanaka-Webster scalar curvature and assume that $F$ is a subcritical function with exponent $3$. Let $u$ be a positive solution to (1.1) such that $u$ tends to zero at infinity.  If there exists a constant $2<\kappa\leq 3$ such that 
          \begin{align}
              tF^{'}(t)-\kappa F(t)\geq 0,\ \forall t>0,
          \end{align} 
          then $(M^3,HM,J,\theta) $ is isometric to $\mathbb{H}^1$ with its standard structures and $u$ is the form of (1.3).
      \end{theorem}
      \begin{theorem}
          Let $(M^5,HM,J,\theta)$ be a 5-dimensional complete noncompact Sasakian manifold with $Ric_b\geq 0 $ and 
           \begin{align*}
               V(R)\leq C R^{6}
           \end{align*}
          for some $C > 0$ and $R > 0$ large enough. Assume that $F$ is a subcritical function with exponent $2$. Let $u$ be a positive solution to (1.1) such that $u$ tends to zero at infinity. If there exist constants $\frac{3}{2}<\kappa\leq 2 $ and $\sigma<2$ such that 
          \begin{align}
              tF^{'}(t)-\kappa F(t)\geq 0,\ \forall t>0,
          \end{align}
          and 
          \begin{align}
              \int_{A_R} F^3(u)u^{-3}\leq CR^\sigma
          \end{align}
          for $\sigma<2$ and $R>0$ large enough. Then $(M^5,HM,J,\theta) $ is isometric to $\mathbb{H}^2$ with its standard structures and $u$ is the form of (1.3).
      \end{theorem}   
      For higher dimensional case, we give a similar result under the integral condition.
      \begin{theorem}
          Let $(M^{ 2n+1}, HM, J,\theta),\ n\geq 3$ be a complete noncompact Sasakian manifold with $Ric_b\geq 0 $ and 
           \begin{align*}
               V(R)\leq C R^{2n+2}
           \end{align*}
           for some $C > 0$ and $R > 0$ large enough.
           Assume that $F$ is a subcritical function with exponent $1+\frac{2}{n}$. Let $u$ be a positive solution to (1.1) such that $u$ tends to zero at infinity. If there exists constants $\frac{2}{n}<\kappa\leq 1+\frac{2}{n} $
          \begin{align}
              tF^{'}(t)-\kappa F(t)\geq 0,\ \forall t>0,
          \end{align}
          and 
          \begin{align}
              \int_{A_R} F^{n+1}(u)u^{-n-1}\leq CR^2
          \end{align}
          for $R>0$ large enough. Then $(M^{2n+1},HM,J,\theta) $ is isometric to $\mathbb{H}^n$ with its standard structures and $u$ is the form of (1.3).
      \end{theorem}
      We recall the following volume comparison theorem on Sasakian manifolds (cf. \cite{LL}).
         \begin{theorem}
             Let $(M^{ 2n+1,} HM, J,\theta)$ be a complete noncompact Sasakian manifold.
             \\
             (1) If $n=1$ and the Tanaka-Webster scalar curvature is nonnegative, then
             \begin{align*}
                 V(R)\leq C R^{4}, \quad \forall R>1,
             \end{align*}
             for some $C>0.$
             \\
             (2) If $n\geq 2$ and the Tanaka-Webster curvature satisfies
             \begin{align}
                 Ric_b(X,X)\geq R(X,JX,JX,X)\geq 0,\quad \forall X\in HM,
             \end{align}
             then 
             \begin{align*}
               V(R)\leq C R^{2n+2},\quad \forall R>1
           \end{align*}
           for some $C > 0$.
         \end{theorem}
        From Theorem 1.4, we have the following corollary.
        \begin{corollary}
            Let $(M^{ 2n+1}, HM, J,\theta),\ n\geq 2$ be a complete noncompact Sasakian manifold with curvature condition (1.12). Replacing the assumption on $Ric_b$ and on $Vol(B_R)$ by (1.12) in Theorem 1.2 or Theorem 1.3, then $(M^{2n+1},HM,J,\theta) $ is isometric to $\mathbb{H}^n$ with its standard structures and $u$ is the form of (1.3).
        \end{corollary}

        ~\\
         $\mathbf{Notation}$ Throughout the paper, we adopt Einstein summation convention over repeated indices. Since the constant $C$ is not important, it may vary at different occurrences.

         \section{Preliminaries}
         In this section, we first give a brief introduction to pseudo-Hermitian geometry.  We refer to \cite{DT,We,Ta} for a complete overview of the material presented here. A smooth manifold $M$ of real dimension $(2n + 1)$ is said to be a CR manifold if there exists a rank $n$ complex subbundle $H^{1,0}M $ of $TM\otimes \mathbb{C}$ satisfying \begin{equation}
         H^{1,0}M\cap H^{0,1}M=\{0\} ,\ \  [\Gamma(H^{1,0}M),\Gamma(H^{1,0}M)]\subseteq \Gamma(H^{1,0}M)  , 
         \end{equation}
         where $H^{0,1}M=\overline{H^{1,0}M}$. Equivalently, the CR structure may also be described by $(HM,J)$, where $HM=Re\{H^{1,0}M\oplus H^{0,1}M\}$ and $J$ is an almost complex structure on $HM$ such that
         \begin{align*}
             J(X+\overline{X})=\sqrt{-1}(X-\overline{X}),\quad \forall X\in H^{1,0}M.
         \end{align*}
         Since $HM$ is naturally oriented by $J$, $M$ is orientable if and only if there exists a global nowhere vanishing 1-form $\theta$ such that $ker\theta=HM$. Any such $\theta$ is called a pseudo-Hermitian structure on $M$. The Levi form $L_\theta$ of $\theta$ is defined by
          \begin{eqnarray}
            L_\theta(X,Y)=d\theta(X,JY) \nonumber
          \end{eqnarray}
         for any $X,Y\in HM$. If the Levi form $L_\theta$ is positive definite for some choice of $\theta$, we call such $(M,HM,J,\theta) $ a pseudo-Hermitian manifold. Noting that $\theta$ is a contact form, thus there exists a unique Reeb vector field $\xi$ satisfying
         \begin{eqnarray}
           \theta(\xi)=1,\ d\theta(\xi,\cdot)=0.
         \end{eqnarray}
         It leads a decomposition of the tangent bundle $TM=HM \oplus R\xi$, which induces a natural projection $\pi_b:TM\rightarrow HM$. This yields a natural Riemannian metric $g_\theta$, called the Webster metric,
         \begin{eqnarray}
         g_\theta=G_\theta+\theta\otimes \theta ,\nonumber
        \end{eqnarray}
        where $ G_\theta=d\theta(\cdot,J\cdot)$. The volume form of $g_\theta$ is defined by $dV_\theta=\theta\wedge(d\theta)^{n}/n!$, and is often omitted when we integrate. 
        \par
        On a pseudo-Hermitian manifold $(M^{2n+1}, HM, J, \theta)$, a Lipschitz curve $\gamma:[0, l] \rightarrow  M$ is called horizontal if $\gamma^{'}\in H_{\gamma(t)}M$ a.e. in $[0, l]$. By the well-known work of Chow \cite{Ch}, there always exists a horizontal curve joining two points $p, q\in M. $ The Carnot-Carath$\acute{\mathrm{e}}$odory distance is defined as
        \begin{eqnarray}
              d(p,q)= inf\left\{ \int_0^l \sqrt{g_\theta(\gamma^{'},\gamma^{'})} dt \ | \  \gamma \in \Gamma(p,q) \right\},
            \nonumber
           \end{eqnarray}
        where $\Gamma(p,q) $ denotes the set of all horizontal curves joining $p$ and $q$. Clearly, Carnot-Carath$\acute{\mathrm{e}}$odory is indeed a distance function and induces the same topology on $M$ (cf. \cite{ABB,St}).  
        \par
         It is known that there exists a canonical connection $\nabla$ on a pseudo-Hermitian manifold, called the Tanaka-Webster connection, preserving the horizontal bundle, the CR structure and the Webster metric. Moreover, its torsion $T_\nabla$ satisfies
        \begin{align*}
            T_\nabla(X,Y)=2d\theta(X,Y)\xi\ and\ T_\nabla(\xi,JX)+JT_\nabla(\xi,X)=0
        \end{align*}
        for any $X,Y\in HM.$ The pseudo-Hermitian torsion $\tau$ of $\nabla$ is defined by $ \tau(X)=T_\nabla(\xi,X)$ for any $X\in TM.$ We say that $M$ is Sasakian if $\tau=0$. In this paper, we only consider Sasakian manifolds.
        \par
        Assume that $\{\eta_\alpha\}_{\alpha=1}^n$ is a local unitary frame field of $H^{1,0}M$ and let $\{ \theta^\alpha\}_{\alpha=1}^n$ be the dual coframe. Then
        \begin{align*}
            d\theta=2\sqrt{-1} \theta^\alpha\wedge \theta^{\Bar{\alpha}}.
        \end{align*}
        Given a smooth function $f$ on the Sasakian manifold $(M^{2n+1}, HM, J, \theta)$, its differential $df$ and covariant derivatives $ \nabla df$ can be expressed as
        \begin{align*}
            df=&f_0\theta+f_\alpha\theta^\alpha+f_{\bar{\alpha}}\theta^{\bar{\alpha}},
            \\
            \nabla d f=& f_{\alpha\beta}\theta^\alpha\otimes\theta^\beta +f_{\alpha\bar{\beta}}\theta^\alpha \otimes \theta^{\bar{\beta}}+ f_{\bar{\alpha}\beta}\theta^{\bar{\alpha}}\otimes \theta^\beta+f_{\bar{\alpha}\bar{\beta}}\theta^{\bar{\alpha}}\otimes \theta^{\bar{\beta}}
            \\       &+f_{0\alpha}\theta\otimes\theta^\alpha+f_{0\bar{\alpha}}\theta\otimes\theta^{\bar{\alpha}}+f_{\alpha0}\theta^\alpha\otimes\theta+f_{\bar{\alpha}0}\theta^{\bar{\alpha}}\otimes \theta ,\nonumber
        \end{align*}
        where 
        \begin{align*}
         f_0=\xi(f),\ f_\alpha=\eta_\alpha(f),\ f_{\bar{\alpha}}=\eta_{\bar{\alpha}}(f),\ f_{\alpha\beta}=\eta_\beta (\eta_\alpha f)-(\nabla_{\eta_\beta}\eta_\alpha) f,\ f_{0\alpha}=\eta_\alpha (\xi f),
        \end{align*}
         and so on. Then the horizontal gradient of $f$ is given by
        \begin{align*}
            \nabla_bf=f_\alpha \eta_{\bar{\alpha}}+f_{\bar{\alpha}}f_\alpha.
        \end{align*}
        Since $\tau=0$, we have (cf. \cite{DT})
        \begin{align*}
            &f_{\alpha\beta}-f_{\beta\alpha}=0,\quad  f_{\alpha\bar{\beta}}-f_{\bar{\beta}\alpha}=2\sqrt{-1}f_0\delta_\alpha^\beta,
            \\
            &f_{0\alpha}-f_{\alpha 0}=0,\quad f_{\alpha\beta0}-f_{\alpha0\beta}=f_{\alpha\bar{\beta}0}-f_{\alpha0\bar{\beta}}=0,
            \\
            &f_{\alpha\beta\bar{\gamma}}-f_{\alpha\bar{\gamma}\beta}=2\sqrt{-1}\delta_\gamma^\beta f_{\alpha0}+R_{\bar{\mu}\alpha\beta\bar{\gamma}}f_\mu,\quad f_{\alpha\bar{\beta}\bar{\gamma}}-f_{\alpha\bar{\gamma}\bar{\beta}}=0,
        \end{align*}
        where $R_{\bar{\mu}\alpha\beta\bar{\gamma}} $ are the components of curvature of the Tanaka-Webster connection. Moreover, $R_{\alpha\bar{\beta}} = R_{\bar{\mu}\alpha\mu\bar{\beta}}$ are the components of the pseudo-Hermitian Ricci curvature $Ric_b$ and $R=R_{\alpha\bar{\alpha}}$ is the Tanaka-Webster scalar curvature. Finally, we define the horizontal energy density by
        \begin{align*}
            |\nabla_b f|^2=2|\partial f |^2=2f_\alpha f_{\bar{\alpha}}
        \end{align*}
         and the sub-Laplacian of $f$ by
        \begin{align*}
            \Delta_b f=f_{\alpha\bar{\alpha}}+f_{\bar{\alpha}\alpha}.
        \end{align*}
        \par
         Next we are aim to give the Jerison-Lee type inequality for the positive solution $u$ of (1.1). Let $e^f=u^{\frac{1}{n}}$, then we have
        \begin{align*}
            -\Delta_b f= 2n|\partial f|^2+2n\frac{F(u)}{u}=2n|\partial f|^2+2n\frac{F(e^{nf})}{e^{nf}}.
        \end{align*}
        Let $H(t)=\frac{F(e^{nt})}{e^{nt}}$, then $0<H\in C^{1} $. Thus, we obtain
        \begin{align}
            -\Delta_b f= 2n|\partial f|^2+2nH(f).
        \end{align}
       Assuming that $F$ is a subcritical function with exponent $1+\frac{2}{n}$, we have
        \begin{align}
            2H-H^{'}\geq 0.
        \end{align}
         As done in \cite{JL,MO}, we define the following
        \begin{align}
            &D_{\alpha\beta}=f_{\alpha\beta}-2f_\alpha f_\beta, \quad D_\alpha=D_{\alpha\beta}f_{\bar{\beta}},\nonumber
            \\
            &E_{\alpha\bar{\beta}}=f_{\alpha\bar{\beta}}-\frac{1}{n}f_{\gamma\bar{\gamma}}\delta_\alpha^\beta,\quad E_\alpha=E_{\alpha\bar{\beta}}f_\beta,
            \\
            &G_\alpha=\sqrt{-1}f_{0\alpha}-\sqrt{-1}f_0f_\alpha+H(f)f_\alpha+|\partial f|^2f_\alpha ,\nonumber
            \\
            &g=|\partial f|^2+H(f)-\sqrt{-1}f_0. \nonumber
        \end{align}
        Then the equation (2.3) can be written as 
        \begin{align}
            f_{\alpha\bar{\alpha}}=-ng.
        \end{align}
        A direct calculation shows that
        \begin{align}            &E_{\alpha\bar{\beta}}=f_{\alpha\bar{\beta}}+g\delta_\alpha^\beta,\quad E_\alpha=f_{\alpha\bar{\beta}}f_\beta+gf_\alpha,
            \nonumber \\
            &D_\alpha=f_{\alpha\bar{\beta}}f_\beta-2|\partial f|^2f_\alpha, \quad G_\alpha=\sqrt{-1}f_{0\alpha}+gf_\alpha.
        \end{align}
        Moreover, we obtain
        \begin{align}
            |\partial f |_{\bar{\alpha}}^2=& D_{\bar{\alpha}}+E_{\bar{\alpha}}+\bar{g}f_{\bar{\alpha}}-2f_{\bar{\alpha}}H(f),
            \nonumber \\
            g_{\bar{\alpha}}=& D_{\bar{\alpha}}+E_{\bar{\alpha}}+G_{\bar{\alpha}}+(H^{'}(f)-2H(f))f_{\bar{\alpha}},
            \\
            \bar{g}_{\bar{\alpha}}=&D_{\bar{\alpha}}+E_{\bar{\alpha}}-G_{\bar{\alpha}}+2\bar{g}f_{\bar{\alpha}}+(H^{'}(f)-2H(f))f_{{\bar{\alpha}}},\nonumber
            \\
            g_0=&\sqrt{-1}f_{\alpha}G_{\bar{\alpha}}-\sqrt{-1}f_{\bar{\alpha}}G_\alpha+2f_0|\partial f|^2+H^{'}(f)f_0-\sqrt{-1}f_{00}. \nonumber
        \end{align}
         In view of the above observations, we give the following inequality (\cite{Zbq}).
         \begin{lemma}
            Let $(M^{ 2n+1}, HM, J,\theta)$ be a complete noncompact Sasakian manifold with $Ric_b\geq 0 $. Set 
            \begin{align}
                \mathcal{M}=&Re\ \nabla_{\eta_{\bar{\alpha}}}\{e^{2(n-1)f}[(g+3\sqrt{-1}f_0)E_\alpha+(g-\sqrt{-1}f_0)D_\alpha-3\sqrt{-1}f_0G_\alpha  \},
            \end{align}
            then
            \begin{align}
                \mathcal{M}\geq &(|D_{\alpha\beta}|^2+|E_{\alpha\bar{\beta}}|^2)e^{2(n-1)f}H(f)+(|G_\alpha|^2+|D_{\alpha\beta}f_{\bar{\gamma}}+E_{\alpha\bar{\gamma}}f_\beta |^2)e^{2(n-1)f} \nonumber
                \\
                &+s_0(|D_\alpha+G_\alpha|^2+|E_\alpha-G_\alpha|^2)e^{2(n-1)f} \nonumber
                \\
                &+e^{2(n-1)f}|\sqrt{1-s_0}(D_\alpha+G_\alpha)+\frac{1}{2\sqrt{1-s_0}}f_\alpha(H^{'}(f)-2H(f))|^2 \nonumber
                \\
                &+e^{2(n-1)f}|\sqrt{1-s_0}(E_\alpha-G_\alpha)+\frac{1}{2\sqrt{1-s_0}}f_\alpha(H^{'}(f)-2H(f))|^2 \nonumber
                \\
                &-(H^{'}(f)-2H(f))e^{2(n-1)f}|\partial f|^2 \left ((2n-1)H(f)+\frac{H^{'}(f)-2H(f)}{2(1-s_0)}\right) \nonumber
                \\
                &- (2n-1)(H^{'}(f)-2H(f))e^{2(n-1)f}|\partial f|^4 
                -3n (H^{'}(f)-2H(f))e^{2(n-1)f}f_0^2,
            \end{align}
            where $0<s_0<1$ is a constant.
        \end{lemma}
        This is exactly Lemma 2.1 in \cite{Zbq} for $p=0$. As shown in \cite{Zbq}, if there exists a constant $\kappa\in (\frac{4}{n}-3,1+\frac{2}{n}]$ such that 
       \begin{align*}
           tF^{'}(t)-\kappa F(t)\geq 0,\quad \forall t>0,
        \end{align*}
        then we have
       \begin{align}
           H^{'}(f)\geq(n\kappa-n)H(f)\geq 2(2-2n)H(f).
       \end{align}  
       This implies that
       \begin{align*}
           \left((2n-1)H(f)+\frac{H^{'}(f)-2H(f)}{2(1-s_0)}\right)\geq \left(2n-1+\frac{n\kappa-n-2}{2(1-s_0)}\right)H(f)\geq 0
       \end{align*}
        for some $0<s_0< min\{1, \frac{n\kappa-4+3n}{4n-2} \}$. In this case, all coefficients on the right side of (2.10) are positive. It is easy to see that we can find a suitable $s_0$ under the condition in Theorem 1.1, Theorem 1.2 and Theorem 1.3. In the sequel, we assume that $0<s_0< min\{1, \frac{n\kappa-4+3n}{4n-2} \}$, i.e., $H$ satisfies (2.11). In particular, we have $\mathcal{M}\geq 0$.
       \par
       Given $R>0$, there exists a cut-off function $\phi$ such that 
        $$  \left\{
        \begin{array}{rcl}
         &\phi=1    &  in\ B_R, \\
         &0\leq \phi\leq 1,    & in\ B_{2R} ,\\
         &\phi=0,   &  in\ M \backslash B_{2R},  \\
          & |\partial \phi|\leq \frac{C}{R},  & in\ M. 
        \end{array}
        \right.
        $$
        Let $A_R=B_{2R}\setminus B_{R}$. As in \cite{CLMR,CMRW}, we obtain a lemma in terms of the test functions.
       \begin{lemma}
           Let $s>2$ and $ \beta\geq 0$ small enough. Then we have the following estimates 
           \begin{align}
              &\int_M\Psi^{-\beta}\mathcal{M}\phi^s
               \nonumber\\
               \leq &C\left(\int_{A_R}  \Psi^{-\beta}\mathcal{M}\phi^s\right)^{\frac{1}{2}}
               \cdot \left(\frac{1}{R^2}\int_{A_R} e^{2(n-1)f}|g|^2\Psi^{-\beta}\phi^{s-2} \right)^{\frac{1}{2}}
           \end{align}
           and 
           \begin{align}
              \int_M\Psi^{-\beta}\mathcal{M}\phi^s\leq  \frac{C}{R^2}\int_{A_R} e^{2(n-1)f}|g|^2\Psi^{-\beta}\phi^{s-2},
           \end{align}
           where $ \Psi=|g|^2e^{-2f}$.
       \end{lemma}
       \begin{proof}
           By multiplying (2.9) by $\Psi^{-\beta}\phi^s$ and integrating by parts, we have
           \begin{align*}
              &\int_M\Psi^{-\beta}\mathcal{M}\phi^s
               \\
               =&\int_M \  Re \nabla_{\eta_{\bar{\alpha}}}\{e^{2(n-1)f}[(g+3\sqrt{-1}f_0)E_\alpha+(g-\sqrt{-1}f_0)D_\alpha-3\sqrt{-1}f_0G_\alpha ] \}\Psi^{-\beta}\phi^s
               \\
               =&\beta\int_M Re \{e^{2(n-1)f}[(g+3\sqrt{-1}f_0)E_\alpha+(g-\sqrt{-1}f_0)D_\alpha-3\sqrt{-1}f_0G_\alpha]\Psi_{\bar{\alpha}}  \}\Psi^{-\beta-1}\phi^s
               \\
               &-s\int_M Re\{e^{2(n-1)f}[(g+3\sqrt{-1}f_0)E_\alpha+(g-\sqrt{-1}f_0)D_\alpha-3\sqrt{-1}f_0G_\alpha]\phi_{\bar{\alpha}}  \}\Psi^{-\beta}\phi^{s-1}.
           \end{align*}
           From the definition of $g$, we obtain that
           \begin{align*}
               &| (g+3\sqrt{-1}f_0)E_\alpha+(g-\sqrt{-1}f_0)D_\alpha-3\sqrt{-1}f_0G_\alpha|
               \\
               \leq & |D_\alpha+E_\alpha|(|\partial f|^2+H(f))+|f_0|\cdot|2D_\alpha-2E_\alpha+3G_\alpha|.
           \end{align*}
           Since
           \begin{align*}
               &|D_\alpha+E_\alpha|\leq |D_\alpha+G_\alpha|+|E_\alpha-G_\alpha| \leq C\sqrt{\mathcal{M}}e^{-(n-1)f},
               \\
               &|2D_\alpha-2E_\alpha+3G_\alpha|\leq 2|D_\alpha+G_\alpha|+2|E_\alpha-G_\alpha|+|G_\alpha|\leq C\sqrt{\mathcal{M}}e^{-(n-1)f},
               \\
               &|\partial f|^2+H(f)\leq |g|,\quad |f_0|\leq |g|,
           \end{align*}
           we have
           \begin{align}
               | (g+3\sqrt{-1}f_0)E_\alpha+(g-\sqrt{-1}f_0)D_\alpha-3\sqrt{-1}f_0G_\alpha|\leq C|g|\sqrt{\mathcal{M}}e^{-(n-1)f}.
           \end{align}
           Next we compute $\Psi_{\bar{\alpha}}$. By (2.8), we have
           \begin{align*}
               \Psi_{\bar{\alpha}}=&( |g|^2e^{-2f})_{\bar{\alpha}}=e^{-2f}(g_{\bar{\alpha}}\bar{g}+g\bar{g}_{\bar{\alpha}})-2|g|^2e^{-2f}f_{\bar{\alpha}}
               \\
               =&e^{-2f}[\bar{g}(D_{\bar{\alpha}}+E_{\bar{\alpha}}+G_{\bar{\alpha}}+(H^{'}(f)-2H(f))f_{\bar{\alpha}} )
               \\
               &+g(D_{\bar{\alpha}}+E_{\bar{\alpha}}-G_{\bar{\alpha}}+2\bar{g}f_{\bar{\alpha}}+(H^{'}(f)-2H(f))f_{{\bar{\alpha}}} )]-2|g|^2e^{-2f}f_{\bar{\alpha}}
               \\
               =&2e^{-2f}[(D_{\bar{\alpha}}+G_{\bar{\alpha}})(|\partial f|^2+H(f))+(E_{\bar{\alpha}}-G_{\bar{\alpha}})(|\partial f|^2+H(f))+\sqrt{-1}f_0G_{\bar{\alpha}} ]
               \\
               &+2e^{-2f}(H^{'}(f)-2H(f))(|\partial f|^2+H(f))f_{{\bar{\alpha}}}.
           \end{align*}
           Noting that $ 4(1-n)H(f)\leq H^{'}(f)\leq 2H(f)$, we have
           \begin{align*}
               |H^{'}(f)-2H(f)|\leq CH(f).
           \end{align*}
           In particular, we find that 
           \begin{align*}
               |(H^{'}(f)-2H(f))(|\partial f|^2+H(f))f_{{\bar{\alpha}}}|\leq C|g|\sqrt{\mathcal{M}}e^{-(n-1)f}.
           \end{align*}
           Thus 
           \begin{align}
               |\Psi_{\bar{\alpha}}|\leq C|g|\sqrt{\mathcal{M}}e^{-(n+1)f}.
           \end{align}
           From (2.14) and (2.15), we obtain 
           \begin{align*}
               &\int_M\Psi^{-\beta}\mathcal{M}\phi^s 
               \\
               \leq& C\beta \int_M \mathcal{M}|g|^2e^{-2f}\Psi^{-\beta-1}\phi^s+Cs\int_M |g|\sqrt{\mathcal{M}}e^{(n-1)f}\Psi^{-\beta}|\partial \phi|\phi^{s-1}
               \\
               =&C\beta\int_M\Psi^{-\beta}\mathcal{M}\phi^s+Cs\int_M |g|\sqrt{\mathcal{M}}e^{(n-1)f}\Psi^{-\beta}|\partial \phi|\phi^{s-1}.
           \end{align*}
           Choosing $0\leq \beta<\frac{1}{C}$, we conclude that 
           \begin{align*}
               \int_M\Psi^{-\beta}\mathcal{M}\phi^s\leq& C\int_M |g|\sqrt{\mathcal{M}}e^{(n-1)f}\Psi^{-\beta}|\partial \phi|\phi^{s-1}
               \\
               \leq & C\left( \int_{A_R}\Psi^{-\beta}\mathcal{M}\phi^s \right)^{\frac{1}{2}} \cdot\left( \frac{1}{R^2}\int_{A_R}e^{2(n-1)f}|g|^2\Psi^{-\beta}\phi^{s-2}\right)^{\frac{1}{2}}.
            \end{align*}
             The inequality (2.13) is a direct consequence of (2.12).
       \end{proof}
       It is easy to see that we need to prove the following inequality
       \begin{align*}
           \int_{A_R}e^{2(n-1)f}|g|^2\Psi^{-\beta}\phi^{s-2}\leq C
       \end{align*}
      for $R>0$ large enough. Finally, we give a result which will be used in the proof of Theorem 1.1.
       \begin{corollary}(\cite{CMRW})
           Let $(M^{2n+1},\theta,J)$ be a $(2n+1)$-dimensional complete Sasakian manifold satisfying the curvature assumptions in Theorem 1.4. Let $u$ be a positive superharmonic function in $M$, i.e. $u \in C^2(M)$ and
           \begin{align*}
            \Delta_b u\leq 0.
            \end{align*}
            Then there exists a constant $C > 0$ such that
            \begin{align*}
                u(x)\geq \frac{C}{r(x)^{max\{3,n+1 \}}}
            \end{align*}
        for any $x \in M$ with $r(x)>1$ and $r(x)$ is the Carnot-Carath$\acute{e}$odory distance from $x$ to some fix piont.
       \end{corollary}

       \section{Proof of Theorem 1.1}
       Suppose that $(M^{2n+1},J,\theta)$ is a complete noncompact Sasakian manifold with $Ric_b \geq 0$. Let $f$, $g$, $H$, $\phi$ and $\mathcal{M}$ be defined as in Section 2. In this section, we assume that there exist constants $1+\frac{1}{n}<\kappa\leq 1+\frac{2}{n}$ such that 
          \begin{align}
              tF^{'}(t)-\kappa F(t)\geq 0,\ \forall t>0. \nonumber
          \end{align} 
          In particular, we have
          \begin{align*}
              H^{'}(f)\geq (n\kappa-n)H(f)>H(f).
          \end{align*}
          First, we give some integral estimates which are useful in the proof.
      \begin{lemma}
          Let $n\geq 1$ and $s\geq 4$. If $\beta\geq  0$ is small enough, there exist a constant $C$ and a small $\epsilon>0$ such that 
          \begin{align}
              \int_M\Psi^{-\beta}\mathcal{M}\phi^s\leq \frac{C}{R^4}\int_{A_R}e^{2(n-1)f}|\partial f|^2\Psi^{-\beta}\phi^{s-4}+\frac{C}{R^2}\int_{A_R}e^{2(n-1)f}|\partial f|^4\Psi^{-\beta}\phi^{s-2}
          \end{align}
          and 
          \begin{align}
              \int_{A_R}e^{2(n-1)f}|g|^2\Psi^{-\beta}\phi^{s-2}
              \leq& \frac{C}{R^2}\int_{A_R}e^{2(n-1)f}|\partial f|^2\Psi^{-\beta}\phi^{s-4}+\epsilon R^2 \int_M\Psi^{-\beta}\mathcal{M}\phi^s
              \nonumber\\
              &+C\int_{A_R}e^{2(n-1)f}|\partial f|^4\Psi^{-\beta}\phi^{s-2}
          \end{align}
          for any $R>0$ .
      \end{lemma}
      \begin{proof}
           From (2.13) and integrating by parts, we have
           \begin{align}
               &\int_M\Psi^{-\beta}\mathcal{M}\phi^s 
               \nonumber\\
               \leq & \frac{C}{R^2}\int_{A_R} e^{2(n-1)f}|g|^2\Psi^{-\beta}\phi^{s-2} 
               \nonumber\\
               =&-\frac{C}{R^2}\int_{A_R} e^{2(n-1)f}Re(f_{\alpha\bar{\alpha}}\bar{g})\Psi^{-\beta}\phi^{s-2}
               \\
               =&\frac{C}{R^2}\int_{A_R} \{ e^{2(n-1)f}Re(f_{\alpha}\bar{g}_{\bar{\alpha}})\Psi^{-\beta}\phi^{s-2}+(s-2) e^{2(n-1)f}Re(f_{\alpha}\phi_{\bar{\alpha}}\bar{g})\Psi^{-\beta}\phi^{s-3}
               \nonumber\\
               &+2(n-1) e^{2(n-1)f}|\partial f|^2Re(\bar{g})\Psi^{-\beta}\phi^{s-2}-\beta e^{2(n-1)f}Re(f_{\alpha}\Psi_{\bar{\alpha}}\bar{g})\Psi^{-\beta-1}\phi^{s-2}\}.\nonumber
           \end{align}
          Similar to get (2.14), we obtain
          \begin{align*}
              |D_{\bar{\alpha}}+E_{\bar{\alpha}}-G_{\bar{\alpha}}|\leq C\sqrt{\mathcal{M}}e^{-(n-1)f}.
          \end{align*}
          Thus
      \begin{align*}
          Re(f_{\alpha}\bar{g}_{\bar{\alpha}})=&Re( (D_{\bar{\alpha}}+E_{\bar{\alpha}}-G_{\bar{\alpha}})f_\alpha+2\bar{g}|\partial f|^2+(H^{'}(f)-2H(f))|\partial f|^2)
          \\
          \leq & C|\partial f|\sqrt{\mathcal{M}}e^{-(n-1)f}+4|\partial f|^4+\frac{1}{2}H^2(f).
      \end{align*} 
      Using Young’s inequality, we have
      \begin{align*}
          &\int_{A_R}  e^{2(n-1)f}Re(f_{\alpha}\bar{g}_{\bar{\alpha}})\Psi^{-\beta}\phi^{s-2}
          \\
          \leq & \int_{A_R}  e^{2(n-1)f}( C|\partial f|\sqrt{\mathcal{M}}e^{-(n-1)f}+4|\partial f|^4+\frac{1}{2}H^2(f))\Psi^{-\beta}\phi^{s-2}
          \\
          \leq & \frac{\epsilon R^2}{2}\int_M\Psi^{-\beta}\mathcal{M}\phi^s+\frac{C}{R^2}\int_{A_R}e^{2(n-1)f}|\partial f|^2\Psi^{-\beta}\phi^{s-4}
          \\
          &+\frac{1}{2}\int_{A_R}e^{2(n-1)f}H^2(f)\Psi^{-\beta}\phi^{s-2}+4\int_{A_R}e^{2(n-1)f}|\partial f|^4\Psi^{-\beta}\phi^{s-2}.
      \end{align*}
      Similarly, we find that
      \begin{align*}
          &(s-2)\int_{A_R}e^{2(n-1)f}Re(f_{\alpha}\phi_{\bar{\alpha}}\bar{g})\Psi^{-\beta}\phi^{s-3}
          \\
          \leq & \epsilon \int_{A_R}e^{2(n-1)f}|g|^2\Psi^{-\beta}\phi^{s-2}+\frac{C}{R^2}\int_{A_R}e^{2(n-1)f}|\partial f|^2\Psi^{-\beta}\phi^{s-4},
          \\
          &2(n-1)\int_{A_R}e^{2(n-1)f}|\partial f|^2Re(\bar{g})\Psi^{-\beta}\phi^{s-2}
          \\
          =&2(n-1)\int_{A_R}e^{2(n-1)f}|\partial f|^4\Psi^{-\beta}\phi^{s-2}+2(n-1)\int_{A_R}e^{2(n-1)f}|\partial f|^2H(f)\Psi^{-\beta}\phi^{s-2}
          \\
          \leq & \epsilon\int_{A_R}e^{2(n-1)f}H^2(f)\Psi^{-\beta}\phi^{s-2}+C\int_{A_R}e^{2(n-1)f}|\partial f|^4\Psi^{-\beta}\phi^{s-2}.
      \end{align*}
      Moreover, from (2.15) and $\Psi=|g|^2e^{-2f}$, we have
      \begin{align*}
          &-\beta\int_{A_R}e^{2(n-1)f}Re(f_{\alpha}\Psi_{\bar{\alpha}}\bar{g})\Psi^{-\beta-1}\phi^{s-2}
          \\
          \leq &C\int_{A_R}e^{(n-3)f}|\partial f||g|^2\sqrt{\mathcal{M}}\Psi^{-\beta-1}\phi^{s-2}
          \\
          \leq &\frac{\epsilon R^2}{2}\int_M\Psi^{-\beta}\mathcal{M}\phi^s+\frac{C}{R^2}\int_{A_R}e^{2(n-1)f}|\partial f|^2\Psi^{-\beta}\phi^{s-4}.
      \end{align*}
      Hence, we derive that 
      \begin{align*}
          &\int_{A_R} e^{2(n-1)f}|g|^2\Psi^{-\beta}\phi^{s-2}
          \\
          \leq& \epsilon R^2 \int_M\Psi^{-\beta}\mathcal{M}\phi^s+\frac{C}{R^2}\int_{A_R}e^{2(n-1)f}|\partial f|^2\Psi^{-\beta}\phi^{s-4}
          \\
          &+(\frac{1}{2}+\epsilon)\int_{A_R}e^{2(n-1)f}H^2(f)\Psi^{-\beta}\phi^{s-2}+C\int_{A_R}e^{2(n-1)f}|\partial f|^4\Psi^{-\beta}\phi^{s-2}.
      \end{align*}
      Notice that
      \begin{align*}
          \int_{A_R} e^{2(n-1)f}|g|^2\Psi^{-\beta}\phi^{s-2}\geq \int_{A_R}e^{2(n-1)f}H^2(f)\Psi^{-\beta}\phi^{s-2},
      \end{align*}
      since $ |g|^2\geq H^2(f)$. Therefore, we obtain (3.2) by choosing a small $\epsilon$. Combining (3.2) and (3.3), we have
      \begin{align*}
          &\int_M\Psi^{-\beta}\mathcal{M}\phi^s
          \leq \frac{C}{R^2}\int_{A_R} e^{2(n-1)f}|g|^2\Psi^{-\beta}\phi^{s-2}
          \\
          \leq &\epsilon C\int_M\Psi^{-\beta}\mathcal{M}\phi^s+\frac{C}{R^4}\int_{A_R}e^{2(n-1)f}|\partial f|^2\Psi^{-\beta}\phi^{s-4}
          \\
          &+\frac{C}{R^2}\int_{A_R}e^{2(n-1)f}|\partial f|^4\Psi^{-\beta}\phi^{s-2}.
      \end{align*}
      Similarly, we obtain (3.1) by choosing a small $\epsilon$.
        
      \end{proof}
      \begin{lemma}
          Let $n=1 $ or  $n=2$. If $s>0$ is large enough and $\beta>0$ is small enough, then there exists a constant $C>0$ such that 
          \begin{align*}
              \int_{A_R}e^{2(n-1)f}|\partial f|^2\Psi^{-\beta}\phi^{s-4}\leq \frac{C}{R^{2(n-\beta)}}Vol(B_{2R})
          \end{align*}
          for $R>0$ large enough.
      \end{lemma}
      \begin{proof}
          Note that 
          \begin{align}
              \Psi^{-\beta}=|g|^{-2\beta}e^{2\beta f}=e^{2\beta f}(|\partial f|^4+H^2(f)+|f_0|^2 )^{-\beta}\leq e^{2\beta f}H^{-2\beta}(f).
          \end{align}
          Then we have
          \begin{align*}
              &\int_{A_R}e^{2(n-1)f}|\partial f|^2\Psi^{-\beta}\phi^{s-4}
              \\
              \leq& \int_{A_R}e^{2(n-1)f}|\partial f|^2e^{2\beta f}H^{-2\beta}(f)\phi^{s-4}
              \\
              =&\int_{A_R}e^{2(n+\beta-1)f}(Re(g)-H(f))H^{-2\beta}(f)\phi^{s-4}
              \\
              =&-\frac{1}{n}\int_{A_R}e^{2(n+\beta-1)f}Re(f_{\alpha\bar{\alpha}})H^{-2\beta}(f)\phi^{s-4}-\int_{A_R}e^{2(n+\beta-1)f}H^{1-2\beta}(f)\phi^{s-4}
              \\
              =&\frac{2(n+\beta-1)}{n}\int_{A_R}e^{2(n+\beta-1)f}|\partial f|^2H^{-2\beta}(f)\phi^{s-4}
              \\
              &-\frac{2\beta}{n}\int_{A_R}e^{2(n+\beta-1)f}|\partial f|^2H^{-1-2\beta}(f)H^{'}\phi^{s-4}
              \\
              &+\frac{s-4}{n}\int_{A_R}e^{2(n+\beta-1)f}Re(f_{\alpha}\phi_{\bar{\alpha}})H^{-2\beta}(f)\phi^{s-5}-\int_{A_R}e^{2(n+\beta-1)f}H^{1-2\beta}(f)\phi^{s-4}
              \\
              \leq &\frac{2(n+\beta-1)-2(n\kappa-n)\beta}{n}\int_{A_R}e^{2(n+\beta-1)f}|\partial f|^2H^{-2\beta}(f)\phi^{s-4}
              \\
              &+\frac{s-4}{n}\int_{A_R}e^{2(n+\beta-1)f}Re(f_{\alpha}\phi_{\bar{\alpha}})H^{-2\beta}(f)\phi^{s-5}-\int_{A_R}e^{2(n+\beta-1)f}H^{1-2\beta}(f)\phi^{s-4},
          \end{align*}
          since $H^{'}(f)\geq (n\kappa-n)H(f)>H(f) $. Moreover, for small $\beta$, we have
          \begin{align*}
              \frac{2(n+\beta-1)-2(n\kappa-n)\beta}{n}<1
          \end{align*}
          whenever $n=1$ or $n=2$. Hence,
          \begin{align*}
              &\int_{A_R}e^{2(n+\beta-1)f}|\partial f|^2H^{-2\beta}(f)\phi^{s-4}
              \\
              \leq &C\int_{A_R}e^{2(n+\beta-1)f}|\partial f|\cdot|\partial \phi|H^{-2\beta}(f)\phi^{s-5}-C\int_{A_R}e^{2(n+\beta-1)f}H^{1-2\beta}(f)\phi^{s-4}
              \\
              \leq & \epsilon \int_{A_R}e^{2(n+\beta-1)f}|\partial f|^2H^{-2\beta}(f)\phi^{s-4}+\frac{C}{R^2}\int_{A_R}e^{2(n+\beta-1)f}H^{-2\beta}(f)\phi^{s-6}
              \\
              &-C\int_{A_R}e^{2(n+\beta-1)f}H^{1-2\beta}(f)\phi^{s-4}.
          \end{align*}
          This implies that 
          \begin{align*}
              &\int_{A_R}e^{2(n+\beta-1)f}|\partial f|^2H^{-2\beta}(f)\phi^{s-4}
              \\
              \leq & \frac{C}{R^2}\int_{A_R}e^{2(n+\beta-1)f}H^{-2\beta}(f)\phi^{s-6}
              -C\int_{A_R}e^{2(n+\beta-1)f}H^{1-2\beta}(f)\phi^{s-4}.
          \end{align*}
          According to the fact that $H^{'}(f)\leq 2H(f)$, we can assume that there exists a $f_0$ such that for any $f\leq f_0$ 
          \begin{align*}
              H(f)\geq \frac{H(f_0)}{e^{2f_0}}e^{2f}=Ce^{2f}.
          \end{align*}
          On the other hand, $u$ tends to zero at infinity. In this case, there exists a $R_0$ such that if $R\geq R_0$, we have $f\leq f_0 $. Hence, we get $ H(f)\geq Ce^{2f}$ for $R$ large enough. Let $a=\frac{n+\beta-1}{n-\beta}$, then
          \begin{align*}
              e^{2(n+\beta-1)f}H^{-2\beta}(f)\leq C e^{2(n+\beta-1-a)f}H^{a-2\beta}(f).
          \end{align*}
          Using Young’s inequality, we obtain
          \begin{align*}
              &\int_{A_R}e^{2(n+\beta-1)f}|\partial f|^2H^{-2\beta}(f)\phi^{s-4}
              \\
              \leq & \frac{C}{R^2}\int_{A_R}e^{2(n+\beta-1-a)f}H^{a-2\beta}(f)\phi^{s-6}
              -C\int_{A_R}e^{2(n+\beta-1)f}H^{1-2\beta}(f)\phi^{s-4}
              \\
              \leq & \frac{C}{R^{\frac{2(1-2\beta)}{1-a}}}\int_{A_R}\phi^{s-4-\frac{2(1-2\beta)}{1-a}} +(\delta-C) \int_{A_R}e^{2(n+\beta-1)f}H^{1-2\beta}(f)\phi^{s-4}.
          \end{align*}
         Letting $s>4+\frac{2(1-2\beta)}{1-a}$ and $\delta$ sufficiently small, we have
         \begin{align}
           \int_{A_R}e^{2(n+\beta-1)f}|\partial f|^2H^{-2\beta}(f)\phi^{s-4}\leq \frac{C}{R^{\frac{2(1-2\beta)}{1-a}}} Vol(B_{2R}).
         \end{align}
         This complete the proof since
         \begin{align*}
             \int_{A_R}e^{2(n-1)f}|\partial f|^2\Psi^{-\beta}\phi^{s-4}\leq \int_{A_R}e^{2(n+\beta-1)f}|\partial f|^2H^{-2\beta}(f)\phi^{s-4}.
         \end{align*}
      \end{proof}
      \begin{lemma}
          Let $n=1$ and $s$, $\beta$ be as in Lemma 3.2. Then there exist constants $C>0$ and $\epsilon>0$ such that
          \begin{align*}
              \int_{A_R}|\partial f|^4\Psi^{-\beta}\phi^{s-2}\leq \epsilon R^2\int_{A_R}\mathcal{M}\Psi^{-\beta}\phi^{s}+\frac{C}{R^2}Vol(B_{2R})
          \end{align*}
          for $R>0$ large enough.
        \end{lemma}
          \begin{proof}
             From (3.4) and $n=1$, we have
             \begin{align}
                &  \int_{A_R}|\partial f|^4\Psi^{-\beta}\phi^{s-2}
                 \nonumber\\
                 \leq &\int_{A_R}|\partial f|^4 e^{2\beta f}H^{-2\beta}(f) \phi^{s-2}
                 \\
                 = & -\int_{A_R}|\partial f|^2Re(f_{\alpha\bar{\alpha}}) e^{2\beta f}H^{-2\beta}(f) \phi^{s-2}-\int_{A_R}|\partial f|^2 e^{2\beta f}H^{1-2\beta}(f) \phi^{s-2}
                 \nonumber\\
                 =& 2\beta \int_{A_R}|\partial f|^4e^{2\beta f}H^{-2\beta}(f) \phi^{s-2}-2\beta \int_{A_R}|\partial f|^4 e^{2\beta f}H^{-1-2\beta}(f)H^{'} \phi^{s-2}
                 \nonumber\\
                 &+(s-2)\int_{A_R}|\partial f|^2Re(f_{\alpha}\phi_{\bar{\alpha}}) e^{2\beta f}H^{-2\beta}(f) \phi^{s-3}-\int_{A_R}|\partial f|^2 e^{2\beta f}H^{1-2\beta}(f) \phi^{s-2}
                 \nonumber\\
                 &+\int_{A_R}Re(f_{\alpha}|\partial f|^2_{\bar{\alpha}}) e^{2\beta f}H^{-2\beta}(f) \phi^{s-2} .\nonumber
             \end{align}
             By (2.8) and $|D_{\bar{\alpha}}+E_{\bar{\alpha}}|\leq C\sqrt{\mathcal{M}}$, we compute that
             \begin{align*}
                 &\int_{A_R}Re(f_{\alpha}|\partial f|^2_{\bar{\alpha}}) e^{2\beta f}H^{-2\beta}(f) \phi^{s-2}
                 \\
                 = & \int_{A_R}Re(f_{\alpha}(D_{\bar{\alpha}}+E_{\bar{\alpha}})) e^{2\beta f}H^{-2\beta}(f) \phi^{s-2}+\int_{A_R}|\partial f|^4 e^{2\beta f}H^{-2\beta}(f) \phi^{s-2}
                 \\
                 &-\int_{A_R}|\partial f|^2 e^{2\beta f}H^{1-2\beta}(f) \phi^{s-2}
                 \\
                 \leq &C\int_{A_R}|\partial f|\sqrt{\mathcal{M}} e^{2\beta f}H^{-2\beta}(f) \phi^{s-2}+\int_{A_R}|\partial f|^4 e^{2\beta f}H^{-2\beta}(f) \phi^{s-2}
                 \\
                 &-\int_{A_R}|\partial f|^2 e^{2\beta f}H^{1-2\beta}(f) \phi^{s-2}.
             \end{align*}
             Since $ H^{'}(f)\geq (\kappa-1)H(f)>H(f)$ and $2<\kappa\leq 3$, we obtain
             \begin{align*}
                 &2\beta(\kappa-2)\int_{A_R}|\partial f|^4 e^{2\beta f}H^{-2\beta}(f) \phi^{s-2}
                 \\
                 \leq &(s-2)\int_{A_R}|\partial f|^2Re(f_{\alpha}\phi_{\bar{\alpha}}) e^{2\beta f}H^{-2\beta}(f) \phi^{s-3}+C\int_{A_R}|\partial f|\sqrt{\mathcal{M}} e^{2\beta f}H^{-2\beta}(f) \phi^{s-2}.
             \end{align*}
             By Young's inequality, we have
             \begin{align}
                 &2\beta(\kappa-2-\epsilon)\int_{A_R}|\partial f|^4 e^{2\beta f}H^{-2\beta}(f) \phi^{s-2}
                 \nonumber\\
                 \leq & \frac{C}{R^2}\int_{A_R}|\partial f|^2 e^{2\beta f}H^{-2\beta}(f) \phi^{s-4}+C\int_{A_R}|\partial f|\sqrt{\mathcal{M}} e^{2\beta f}H^{-2\beta}(f) \phi^{s-2}
                 \nonumber\\
                 = & \frac{C}{R^2}\int_{A_R}|\partial f|^2 e^{2\beta f}H^{-2\beta}(f) \phi^{s-4}+C\int_{A_R}|\partial f|\sqrt{\mathcal{M}} \Psi^{-\beta}|g|^{2\beta}H^{-2\beta}(f) \phi^{s-2}
                 \nonumber\\
                 \leq & \frac{C}{R^2}\int_{A_R}|\partial f|^2 e^{2\beta f}H^{-2\beta}(f) \phi^{s-4}+\epsilon R^2\int_{A_R}\mathcal{M}\Psi^{-\beta} \phi^{s}
                 \nonumber\\
                  & +\frac{C}{R^2} \int_{A_R}H^{-4\beta}(f)|g|^{4\beta}|\partial f|^2\Psi^{-\beta} \phi^{s-4},
             \end{align}
             where $\epsilon$ is a small constant. Using Young’s inequality again, we derive that 
             \begin{align}
                 &H^{-4\beta}(f)|g|^{4\beta}|\partial f|^2\Psi^{-\beta}\phi^{s-4}
                 \nonumber\\
                 \leq & 2\beta |g|^2\Psi^{-\beta}\phi^{s-2}+(1-2\beta)H^{-\frac{4\beta}{1-2\beta}}(f)|\partial f|^{\frac{2}{1-2\beta}}\Psi^{-\beta}\phi^{s-\frac{4(1-\beta)}{1-2\beta}}
                 \nonumber\\
                 \leq &2\beta |g|^2\Psi^{-\beta}\phi^{s-2}+\frac{1}{2}|\partial f|^4H^{-8\beta}(f)\Psi^{-\beta-\beta(1-4\beta)}\phi^{s-2}
                 +\frac{1-4\beta}{2}\phi^{s-\frac{6-8\beta}{1-4\beta}}.
             \end{align}
            As discussed in Lemma 3.2, we have $ H(f)\geq Ce^{2f}$ for $R$ large enough. Together with Corollary 2.3 and (3.4), we get
            \begin{align}
                H^{-8\beta}(f)\Psi^{-\beta(1-4\beta)}\leq &C H^{-8\beta}(f) e^{2\beta(1-4\beta)f}H^{-2\beta(1-4\beta)}(f)
                \nonumber
                \\
                \leq & Ce^{-18\beta f+8\beta^2f}\leq CR^{54\beta}
            \end{align}
            by choosing $\beta<\frac{1}{4}$. Substituting (3.8) and (3.9) into (3.7), we have
            \begin{align*}
                &\int_{A_R}|\partial f|^4 e^{2\beta f}H^{-2\beta}(f) \phi^{s-2}
                \\
                \leq & \frac{C}{R^2}\int_{A_R}|\partial f|^2 e^{2\beta f}H^{-2\beta}(f) \phi^{s-4}+\epsilon R^2\int_{A_R}\mathcal{M}\Psi^{-\beta} \phi^{s}
                \\
                &+\frac{C}{R^{2-54\beta}}\int_{A_R}|\partial f|^4 \Psi^{-\beta}\phi^{s-2}+\frac{C}{R^2}\int_{A_R}\phi^{s-\frac{6-8\beta}{1-4\beta}}+\frac{C}{R^2}\int_{A_R}|g|^2\Psi^{-\beta}\phi^{s-2}
                \\
                \leq &\epsilon R^2\int_{A_R}\mathcal{M}\Psi^{-\beta}+\frac{C}{R^{2-54\beta}}\int_{A_R}|\partial f|^4 \Psi^{-\beta}\phi^{s-2}+\frac{C}{R^2}Vol(B_{2R}),
            \end{align*}
            where we use (3.5), Lemma 3.1 and Lemma 3.2. Then the result follows from (3.6) for large $R$ and $\beta<\frac{1}{27}$.
            \end{proof}
            Now we are in the place to proof the Theorem 1.1.

            ~\\
      $\mathbf{Proof\ of\ Theorem\ 1.1.}$ Let $f$, $g$, $H$, $\phi$ and $\mathcal{M}$ be defined as in above. Assume that $R>0$ large enough such that $H(f)\geq Ce^{2f}$. From Lemma 3.1, Lemma 3.2 and Lemma 3.3, we have
      \begin{align*}
          &\int_M\Psi^{-\beta}\mathcal{M}\phi^s\leq \frac{C}{R^4}\int_{A_R}|\partial f|^2\Psi^{-\beta}\phi^{s-4}+\frac{C}{R^2}\int_{A_R}|\partial f|^4\Psi^{-\beta}\phi^{s-2}
          \\
          \leq & \frac{C}{R^{6-2\beta}}Vol(B_{2R})+C\epsilon \int_M\Psi^{-\beta}\mathcal{M}\phi^s+\frac{C}{R^4}Vol(B_{2R}).
      \end{align*}
      Choosing $\epsilon$ small enough and using Theorem 1.4, we obtain
      \begin{align*}
          \int_M\Psi^{-\beta}\mathcal{M}\phi^s\leq C.
      \end{align*}
       Combining with (3.2), we have
      \begin{align*}
          \int_{A_R}|g|^2\Psi^{-\beta}\phi^{s-2}\leq CR^2.
      \end{align*}
       Then Lemma 2.2 yields that
       \begin{align*}
           \int_M\Psi^{-\beta}\mathcal{M}\phi^s\leq C\left(\int_{A_R}\Psi^{-\beta}\mathcal{M}\phi^s \right)^{\frac{1}{2}}.
       \end{align*}
       Hence
       \begin{align*}
           \int_M\Psi^{-\beta}\mathcal{M}=0,
       \end{align*}
       i.e. $\mathcal{M}=0 $. By Lemma 2.1, we have $|\partial f|=0 $ or $H^{'}(f)=2H(f)$. Since $f$ satisfies (2.3) and $0<H\in C^{1} $, we have $H^{'}(f)=2H(f)$. In particular,
       \begin{align*}
           F(u)=Cu^{1+\frac{2}{n}}.
       \end{align*}
       The conclusion then follows from \cite{CMRW}. \qed

       \section{Proof of Theorem 1.2}
       In this section, we assume that $n=2$. Let $f$, $g$, $H$, $\phi$ and $\mathcal{M}$ be defined as in Section 2. In this section, we assume that there exist constants $\frac{3}{2}<\kappa\leq 2$ such that 
          \begin{align}
              tF^{'}(t)-\kappa F(t)\geq 0,\ \forall t>0. \nonumber
          \end{align} 
          In particular, we have
          \begin{align*}
              H^{'}(f)\geq (2\kappa-2)H(f)>H(f).
          \end{align*}
          First, we give some integral estimates.
          \begin{lemma}
              Let $s>0$ be large enough and $\beta>0$ be small enough. Then there exist a constant $C>0$ and a small constant $\epsilon>0$ such that the following inequalities hold
              \begin{align}
                  &\int_{A_R}e^{2f}H(f)|\partial f|^2\Psi^{-\beta}\phi^{s-2}
                  \nonumber\\
                  \leq &\frac{C}{R^2}\int_{A_R}e^{2(1-\beta)f}H(f)\phi^{s-4}+C\int_{A_R}e^{2(1-\beta)f}H^2(f)\phi^{s-2}
              \end{align}
              and 
              \begin{align}
                  &\int_{A_R}e^{2f}|\partial f|^4\Psi^{-\beta}\phi^{s-2}
                  \nonumber\\
                  \leq &\frac{C}{R^2}\int_{A_R}e^{2f}|\partial f|^2\Psi^{-\beta}\phi^{s-4}+C\int_{A_R}e^{2f}H(f)|\partial f|^2\Psi^{-\beta}\phi^{s-2}
                  \\
                  &+\epsilon R^2\int_{A_R}\Psi^{-\beta}\mathcal{M}\phi^s \nonumber
              \end{align}
              for $R>0$ large enough.
          \end{lemma}
          \begin{proof}
              Note that $R>0$ large enough, we have $H(f)\geq Ce^{2}f$. From (3.4), we have
              \begin{align}
                  \int_{A_R}e^{2f}H(f)|\partial f|^2\Psi^{-\beta}\phi^{s-2}\leq C\int_{A_R}e^{2(1-\beta)f}H(f)|\partial f|^2\phi^{s-2}.
              \end{align}
              A direct computation shows that
              \begin{align*}
                  &\int_{A_R}e^{2(1-\beta)f}H(f)|\partial f|^2\phi^{s-2}=\int_{A_R}e^{2(1-\beta)f}H(f)(Re(g)-H(f))\phi^{s-2}
                  \\
                  =&-\frac{1}{2}\int_{A_R}e^{2(1-\beta)f}H(f)Re(f_{\alpha\bar{\alpha}})\phi^{s-2}-\int_{A_R}e^{2(1-\beta)f}H^2(f)\phi^{s-2}
                  \\
                  =& (1-\beta)\int_{A_R}e^{2(1-\beta)f}H(f)|\partial f|^2\phi^{s-2}+\frac{s-2}{2}\int_{A_R}e^{2(1-\beta)f}H(f)Re(f_{\alpha}\phi_{\bar{\alpha}})\phi^{s-3}
                  \\
                  &+\frac{1}{2}\int_{A_R}e^{2(1-\beta)f}H^{'}(f)Re(f_{\alpha\bar{\alpha}})\phi^{s-2}-\int_{A_R}e^{2(1-\beta)f}H^2(f)\phi^{s-2}.
              \end{align*}
              Since $ H^{'}(f)>H(f)$, we have
              \begin{align*}
                  &(\frac{1}{2}-\beta)\int_{A_R}e^{2(1-\beta)f}H(f)|\partial f|^2\phi^{s-2}
                  \\
                  \leq& \int_{A_R}e^{2(1-\beta)f}H^2(f)\phi^{s-2}-\frac{s-2}{2}\int_{A_R}e^{2(1-\beta)f}H(f)|\partial f|\cdot|\partial \phi|\phi^{s-3}
                  \\
                  \leq &\epsilon \int_{A_R}e^{2(1-\beta)f}H(f)|\partial f|^2\phi^{s-2}+\frac{C}{R^2}\int_{A_R}e^{2(1-\beta)f}H(f)\phi^{s-4}
                  \\
                  &+C\int_{A_R}e^{2(1-\beta)f}H^2(f)\phi^{s-2}.
              \end{align*}
              Letting $0<\epsilon<\frac{1}{2}-\beta$, we derive that 
              \begin{align*}
                  &\int_{A_R}e^{2(1-\beta)f}H(f)|\partial f|^2\phi^{s-2}
                  \\
                  \leq &\frac{C}{R^2}\int_{A_R}e^{2(1-\beta)f}H(f)\phi^{s-4}
                  +C\int_{A_R}e^{2(1-\beta)f}H^2(f)\phi^{s-2}.
              \end{align*}
              Then (4.1) follows from (4.3). In order to obtain (4.2), we compute that
              \begin{align*}
                  &\int_{A_R}e^{2f}|\partial f|^4\Psi^{-\beta}\phi^{s-2}
                  \\
                  =&\frac{1}{2}\int_{A_R}e^{2f}|\partial f|^2Re(f_{\alpha\bar{\alpha}})\Psi^{-\beta}\phi^{s-2}-\int_{A_R}e^{2f}|\partial f|^2H(f)\Psi^{-\beta}\phi^{s-2}
                  \\
                  =&\int_{A_R}e^{2f}|\partial f|^4\Psi^{-\beta}\phi^{s-2}+\frac{s-2}{2}\int_{A_R}e^{2f}|\partial f|^2Re(f_\alpha\phi_{\bar{\alpha}})\Psi^{-\beta}\phi^{s-3}
                  \\
                  &+\frac{1}{2}\int_{A_R}e^{2f}Re(f_\alpha|\partial f|^2_{\bar{\alpha}})\Psi^{-\beta}\phi^{s-2}-\frac{\beta}{2}\int_{A_R}e^{2f}|\partial f|^2Re(f_\alpha\Psi_{\bar{\alpha}})\Psi^{-\beta-1}\phi^{s-2}
                  \\
                  &-\int_{A_R}e^{2f}|\partial f|^2H(f)\Psi^{-\beta}\phi^{s-2},
              \end{align*}
              that is,
              \begin{align*}
                  0=&\frac{s-2}{2}\int_{A_R}e^{2f}|\partial f|^2Re(f_\alpha\phi_{\bar{\alpha}})\Psi^{-\beta}\phi^{s-3}+\frac{1}{2}\int_{A_R}e^{2f}Re(f_\alpha|\partial f|^2_{\bar{\alpha}})\Psi^{-\beta}\phi^{s-2}
                  \\
                  &-\frac{\beta}{2}\int_{A_R}e^{2f}|\partial f|^2Re(f_\alpha\Psi_{\bar{\alpha}})\Psi^{-\beta-1}\phi^{s-2}-\int_{A_R}e^{2f}|\partial f|^2H(f)\Psi^{-\beta}\phi^{s-2}
                  \\
                  =&\frac{s-2}{2}\int_{A_R}e^{2f}|\partial f|^2Re(f_\alpha\phi_{\bar{\alpha}})\Psi^{-\beta}\phi^{s-3}+\frac{1}{2}\int_{A_R}e^{2f}Re(f_\alpha(D_{\bar{\alpha}}+E_{\bar{\alpha}}))\Psi^{-\beta}\phi^{s-2}
                  \\
                  &-\frac{\beta}{2}\int_{A_R}e^{2f}|\partial f|^2Re(f_\alpha\Psi_{\bar{\alpha}})\Psi^{-\beta-1}\phi^{s-2}-2\int_{A_R}e^{2f}|\partial f|^2H(f)\Psi^{-\beta}\phi^{s-2}
                  \\
                  &+\frac{1}{2}\int_{A_R}e^{2f}|\partial f|^2Re(\bar{g})\Psi^{-\beta}\phi^{s-2}.
              \end{align*}
              Noting that $Re(\bar{g})=|\partial f|^2+H(f)$ and $\Psi^{-2}|\partial f|^4|g|^2\leq e^{4f}$, we have
              \begin{align*}
                  &\frac{1}{2}\int_{A_R}e^{2f}|\partial f|^4\Psi^{-\beta}\phi^{s-2}
                  \\
                  \leq &\epsilon^{'} \int_{A_R}e^{2f}|\partial f|^4\Psi^{-\beta}\phi^{s-2}+\frac{C}{R^2}\int_{A_R}e^{2f}|\partial f|^2\Psi^{-\beta}\phi^{s-4}+\epsilon R^2\int_{A_R}\Psi^{-\beta}\mathcal{M}\phi^s
                  \\
                  &+\frac{C}{R^2}\int_{A_R}e^{-2f}|\partial f|^6|g|^2\Psi^{-\beta-2}\phi^{s-4}+\frac{3}{2}\int_{A_R}e^{2f}|\partial f|^2H(f)\Psi^{-\beta}\phi^{s-2}
                  \\
                  \leq & \epsilon^{'} \int_{A_R}e^{2f}|\partial f|^4\Psi^{-\beta}\phi^{s-2}+\frac{C}{R^2}\int_{A_R}e^{2f}|\partial f|^2\Psi^{-\beta}\phi^{s-4}+\epsilon R^2\int_{A_R}\Psi^{-\beta}\mathcal{M}\phi^s
                  \\
                  &+\frac{3}{2}\int_{A_R}e^{2f}|\partial f|^2H(f)\Psi^{-\beta}\phi^{s-2}.
              \end{align*}
              By choosing $\epsilon^{'}<\frac{1}{2}$, we obtain (4.2).
          \end{proof}
          Now we are ready to prove the Theorem 1.2.
          
       ~\\
      $\mathbf{Proof\ of\ Theorem\ 1.2.}$ Let $f$, $g$, $H$, $\phi$ and $\mathcal{M}$ be defined as in above. Assume that $R>0$ large enough such that $H(f)\geq Ce^{2f}$. From Lemma 3.1 and Lemma 4.1, we have
      \begin{align*}
          &\int_M\Psi^{-\beta}\mathcal{M}\phi^s\leq \frac{C}{R^4}\int_{A_R}e^{2f}|\partial f|^2\Psi^{-\beta}\phi^{s-4}+\frac{C}{R^2}\int_{A_R}e^{2f}|\partial f|^4\Psi^{-\beta}\phi^{s-2}
          \\
          \leq &\frac{C}{R^4}\int_{A_R}e^{2f}|\partial f|^2\Psi^{-\beta}\phi^{s-4}+\epsilon \int_{A_R}\Psi^{-\beta}\mathcal{M}\phi^s+\frac{C}{R^4}\int_{A_R}e^{2(1-\beta)f}H(f)\phi^{s-4}
          \\
          &+\frac{C}{R^2}\int_{A_R}e^{2(1-\beta)f}H^2(f)\phi^{s-2}.
      \end{align*}
      Applying Lemma 3.2 and letting $0<\epsilon<1$, we have
      \begin{align*}
          &\int_M\Psi^{-\beta}\mathcal{M}\phi^s
          \\
          \leq& \frac{C}{R^{8-2\beta}}Vol(B_{2R})+\frac{C}{R^4}\int_{A_R}e^{2(1-\beta)f}H(f)\phi^{s-4}
          +\frac{C}{R^2}\int_{A_R}e^{2(1-\beta)f}H^2(f)\phi^{s-2}.
      \end{align*}
      Noting that $H(f)\geq Ce^{2f}$ and using H$\Ddot{\mathrm{o}} $lder’s inequality, we derive that
      \begin{align*}
          &\int_M\Psi^{-\beta}\mathcal{M}\phi^s
          \\
          \leq& \frac{C}{R^{8-2\beta}}Vol(B_{2R})+\frac{C}{R^4}\int_{A_R}H^{2-\beta}(f)\phi^{s-4}
          +\frac{C}{R^2}\int_{A_R}H^{3-\beta}(f)\phi^{s-2}
          \\
          \leq &\frac{C}{R^{8-2\beta}}Vol(B_{2R})+\frac{C}{R^4}\left(\int_{A_R}H^3(f) \right)^{\frac{3-\beta}{3}}\cdot (Vol(B_{2R}))^{\frac{\beta}{3}}
          \\
          &+\frac{C}{R^2}\left(\int_{A_R}H^3(f) \right)^{\frac{2-\beta}{3}}\cdot (Vol(B_{2R}))^{\frac{1+\beta}{3}}
          \\
          \leq & \frac{C}{R^{2-2\beta}}+\frac{C}{R^{2-\sigma-(2-\frac{\sigma}{3})\beta}}+\frac{C}{R^{2-\frac{2\sigma}{3}-(2-\frac{\sigma}{3})\beta}},
      \end{align*}
      where the last inequality follows from (1.9). Letting $R\rightarrow \infty$, we have
      \begin{align*}
          \int_M\Psi^{-\beta}\mathcal{M}=0,
      \end{align*}
       provided $\beta$ is small enough. In particular, $\mathcal{M} =0$. Similar to the proof of Theorem 1.1, we have $H^{'}(f)=2H(f)$. The conclusion then follows from \cite{CMRW}. \qed

       \section{Proof of Theorem 1.3}
        Let $(M^{ 2n+1}, HM, J,\theta),\ n\geq 3$ be a complete noncompact Sasakian manifold with $Ric_b\geq 0 $ and 
           \begin{align*}
               V(R)\leq C R^{2n+2}
           \end{align*}
           for some $C > 0$ and every $R > 0$ large enough. Let $f$, $g$, $H$, $\phi$ and $\mathcal{M}$ be defined as in Section 2. In this section, we assume that there exist constants $\frac{2}{n}<\kappa\leq 1+\frac{2}{n} $ such that 
          \begin{align}
              tF^{'}(t)-\kappa F(t)\geq 0,\ \forall t>0. \nonumber
          \end{align} 
          In particular, we have
          \begin{align*}
              H^{'}(f)\geq (n\kappa-n)H(f)>(2-n)H(f).
          \end{align*}
          We first state the following inequalities.
          \begin{lemma}
               Let $\epsilon> 0$ small enough and $s > 0$ large enough. There exists a constant $C>0$ such that 
               \begin{align}
                   &\int_{A_R}e^{2(n-1)f}|\partial f|^2\phi^{s-4}\leq C\int_{A_R}e^{2(n-1)f}H(f)\phi^{s-4}+\frac{C}{R^{2n}}Vol(B_{2R}),
                   \\
                   &\int_{A_R}e^{2(n-1)f}H(f)|\partial f|^2\phi^{s-2}\leq C\int_{A_R}e^{2(n-1)f}H^2(f)\phi^{s-2}+\frac{C}{R^2}\int_{A_R}e^{2(n-1)f}H(f)\phi^{s-4}
               \end{align}
               and
               \begin{align}
                   &\int_{A_R}e^{2(n-1)f}|\partial f|^4\phi^{s-2}
                   \nonumber\\
                   \leq&C\int_{A_R}e^{2(n-1)f}H(f)|\partial f|^2\phi^{s-2}+\frac{C}{R^2}\int_{A_R}e^{2(n-1)f}|\partial f|^2\phi^{s-4}
                   +\epsilon R^2 \int_{A_R}\mathcal{M}\phi^s
               \end{align}
               for $R>0$ large enough.
          \end{lemma}
      \begin{proof}
          Observe that
          \begin{align*}
             & \int_{A_R}e^{2(n-1)f}|\partial f|^2\phi^{s-4}=\int_{A_R}e^{2(n-1)f}(Re(g)-H(f))\phi^{s-4}
             \\
             =&-\frac{1}{n}\int_{A_R}e^{2(n-1)f}Re(f_{\alpha\bar{\alpha}})\phi^{s-4}-\int_{A_R}e^{2(n-1)f}H(f)\phi^{s-4}
             \\
             =&\frac{2(n-1)}{n}\int_{A_R}e^{2(n-1)f}|\partial f|^2\phi^{s-4}+\frac{s-4}{n}\int_{A_R}e^{2(n-1)f}Re(f_{\alpha}\phi_{\bar{\alpha}})\phi^{s-5}
             \\
             &-\int_{A_R}e^{2(n-1)f}H(f)\phi^{s-4}.
          \end{align*}
          This yields that 
          \begin{align*}
              &\frac{n-2}{n}\int_{A_R}e^{2(n-1)f}|\partial f|^2\phi^{s-4}
              \\
              \leq&C\int_{A_R}e^{2(n-1)f}|\partial f|\cdot|\partial \phi|\phi^{s-5}+ \int_{A_R}e^{2(n-1)f}H(f)\phi^{s-4}
              \\
              \leq & \epsilon \int_{A_R}e^{2(n-1)f}|\partial f|^2\phi^{s-4}+\frac{C}{R^2}\int_{A_R}e^{2(n-1)f}\phi^{s-6}+\int_{A_R}e^{2(n-1)f}H(f)\phi^{s-4}.
          \end{align*}
          Letting $0<\epsilon<\frac{n-2}{n}$, (5.1) follows from
          \begin{align*}
              &\int_{A_R}e^{2(n-1)f}|\partial f|^2\phi^{s-4}
              \\
              \leq& \frac{C}{R^2}\int_{A_R}e^{2(n-1)f}\phi^{s-6}+C\int_{A_R}e^{2(n-1)f}H(f)\phi^{s-4}
              \\
              \leq & \frac{C}{R^2}\int_{A_R}e^{2(n-1-\frac{n-1}{n})f}H^{\frac{n-1}{n}}(f)\phi^{s-6}+C\int_{A_R}e^{2(n-1)f}H(f)\phi^{s-4}
              \\
              \leq & C\int_{A_R}e^{2(n-1)f}H(f)\phi^{s-4}+\frac{C}{R^{2n}}\int \phi^{s-4-2n}.
          \end{align*}
          Here we assume that $ R>0$ large enough such that $H(f)\geq Ce^{2f}$. Similarly, 
          \begin{align*}
              &\int_{A_R}e^{2(n-1)f}H(f)|\partial f|^2\phi^{s-2}
              \\
              =&\frac{2(n-1)}{n}\int_{A_R}e^{2(n-1)f}H(f)|\partial f|^2\phi^{s-2}+\frac{s-2}{n}\int_{A_R}e^{2(n-1)f}H(f)Re(f_{\alpha}\phi_{\bar{\alpha}})\phi^{s-3}
             \\
             &+\frac{1}{n}\int_{A_R}e^{2(n-1)f}H^{'}(f)|\partial f|^2\phi^{s-2}-\int_{A_R}e^{2(n-1)f}H^2(f)\phi^{s-2}.
          \end{align*}
          Noting that $ H^{'}(f)\geq (n\kappa-n)H(f)>(2-n)H(f)$, we have 
          \begin{align*}
              &\frac{n\kappa-2}{n}\int_{A_R}e^{2(n-1)f}H(f)|\partial f|^2\phi^{s-2}
              \\
              \leq &\int_{A_R}e^{2(n-1)f}H^2(f)\phi^{s-2}+\epsilon \int_{A_R}e^{2(n-1)f}H(f)|\partial f|^2\phi^{s-2}
              \\
              &+\frac{C}{R^2}\int_{A_R}e^{2(n-1)f}H(f)\phi^{s-4}.
          \end{align*}
          By choosing $\epsilon$ small enough, we obtain (5.2). Finally, \begin{align*}
              &\int_{A_R}e^{2(n-1)f}|\partial f|^4\phi^{s-2}
              \\
              =&\frac{2(n-1)}{n}\int_{A_R}e^{2(n-1)f}|\partial f|^4\phi^{s-2}+\frac{s-2}{n}\int_{A_R}e^{2(n-1)f}|\partial f|^2Re(f_{\alpha}\phi_{\bar{\alpha}})\phi^{s-3}
             \\
             &+\frac{1}{n}\int_{A_R}e^{2(n-1)f}Re(f_\alpha|\partial f|^2_{\bar{\alpha}})\phi^{s-2}-\int_{A_R}e^{2(n-1)f}H(f)|\partial f|^2\phi^{s-2}.
          \end{align*}
          A direct calculation implies that
          \begin{align*}
              &\frac{n-2}{n}\int_{A_R}e^{2(n-1)f}|\partial f|^4\phi^{s-2}
              \\
              =&\int_{A_R}e^{2(n-1)f}H(f)|\partial f|^2\phi^{s-2}-\frac{s-2}{n}\int_{A_R}e^{2(n-1)f}|\partial f|^2Re(f_{\alpha}\phi_{\bar{\alpha}})\phi^{s-3}
              \\
              &-\frac{1}{n}\int_{A_R}e^{2(n-1)f}Re(f_\alpha(D_{\bar{\alpha}}+E_{\bar{\alpha}}))\phi^{s-2}-\frac{1}{n}\int_{A_R}e^{2(n-1)f}Re(\bar{g})|\partial f|^2\phi^{s-2}
              \\
              &+\frac{2}{n}\int_{A_R}e^{2(n-1)f}H(f)|\partial f|^2\phi^{s-2}
              \\
              \leq &(1+\frac{1}{n})\int_{A_R}e^{2(n-1)f}H(f)|\partial f|^2\phi^{s-2}+C\int_{A_R}e^{(n-1)f}|\partial f|\sqrt{\mathcal{M}}\phi^{s-2}
              \\
              &-(\frac{1}{n}-\epsilon)\int_{A_R}e^{2(n-1)f}|\partial f|^4\phi^{s-2} +\frac{C}{R^2}\int_{A_R}e^{2(n-1)f}|\partial f|^2\phi^{s-4}
              \\
              \leq &(1+\frac{1}{n})\int_{A_R}e^{2(n-1)f}H(f)|\partial f|^2\phi^{s-2}+\epsilon R^2\int_{A_R}\mathcal{M}\phi^{s}
              \\
              &-(\frac{1}{n}-\epsilon)\int_{A_R}e^{2(n-1)f}|\partial f|^4\phi^{s-2} +\frac{C}{R^2}\int_{A_R}e^{2(n-1)f}|\partial f|^2\phi^{s-4}.
          \end{align*}
          Hence for $\epsilon$ sufficiently small, we deduce (5.3).
      \end{proof}
      Noting that in the proof of Lemma 3.1, we do not use the condition $H^{'}(f)\geq H(f)$. Hence Lemma 3.1 is also valid under the conditions in Theorem 1.3. In this case, we assume that $\beta=0$.

      ~\\
      $\mathbf{Proof\ of\ Theorem\ 1.3.}$ Let $f$, $g$, $H$, $\phi$ and $\mathcal{M}$ be defined as in above. Assume that $R>0$ large enough such that $H(f)\geq Ce^{2f}$. From Lemma 3.1 and Lemma 5.1, we have
      \begin{align*}
          &\int_M\mathcal{M}\phi^s
          \\
          \leq &\frac{C}{R^4}\int_{A_R}e^{2(n-1)f}|\partial f|^2\phi^{s-4}+\frac{C}{R^2}\int_{A_R}e^{2(n-1)f}|\partial f|^4\phi^{s-2}
          \\
          \leq &\frac{C}{R^{2n+4}}Vol(B_{2R})+\frac{C}{R^{4}}\int_{A_R}e^{2(n-1)f}H(f)\phi^{s-4}
          \\
          &+\epsilon \int_M\mathcal{M}\phi^s+\frac{C}{R^{2}}\int_{A_R}e^{2(n-1)f}H^2(f)\phi^{s-2}.
      \end{align*}
      By choosing $\epsilon>0$ small enough and using $H(f)\geq Ce^{2f}$, we find that
      \begin{align*}
          &\int_M\mathcal{M}\phi^s
          \\
          \leq &\frac{C}{R^{2n+4}}Vol(B_{2R})+\frac{C}{R^{4}}\int_{A_R}e^{2(n-1)f}H(f)\phi^{s-4}+\frac{C}{R^{2}}\int_{A_R}e^{2(n-1)f}H^2(f)\phi^{s-2}
          \\
          \leq &\frac{C}{R^{2}}+\frac{C}{R^{4}}\int_{A_R}H^n(f)\phi^{s-4}+\frac{C}{R^{2}}\int_{A_R}H^{n+1}(f)\phi^{s-2}
          \\
          \leq &\frac{C}{R^{2}}+\frac{C}{R^{4}}\left(\int_{A_R}H^{n+1}(f)\right)^{\frac{n}{n+1}}\cdot\left(Vol(B_{2R})\right)^{\frac{1}{n+1}}+\frac{C}{R^{2}}\int_{A_R}H^{n+1}(f)
          \\
          \leq & \frac{C}{R^{2}}+C+\frac{C}{R^{\frac{2}{n+1}}}\leq C.
      \end{align*}
      It follows that 
      \begin{align*}
          \int_M\mathcal{M}\leq C,
      \end{align*}
      which implies that 
      \begin{align*}
          &\int_{A_R}e^{2(n-1)f}|g|^2\Psi^{-\beta}\phi^{s-2}
          \\
              \leq& \frac{C}{R^2}\int_{A_R}e^{2(n-1)f}|\partial f|^2\phi^{s-4}+\epsilon R^2 \int_M\mathcal{M}\phi^s
              +C\int_{A_R}e^{2(n-1)f}|\partial f|^4\phi^{s-2}
              \\
              \leq &CR^2.
      \end{align*}
      Finally, by Lemma 2.2, we obtain
      \begin{align*}
           \int_M\mathcal{M}=0,
       \end{align*}
      that is, $\mathcal{M}=0 $. Hence we deduce the conclusion as in the proof of Theorem 1.1 and Theorem 1.2.  \qed

        \section{ Acknowledgments}
        The author thanks Professor Yuxin Dong and Professor Xiaohua Zhu for their continued support and encouragement.

\bibliographystyle{siam}
\bibliography{ref}

~\\
  Biqiang Zhao
  \\
  $Beijing\ International\ Center\ for$
  \\
  $Mathematical\ Research $
\\
  $ Peking\ University$
\\
   $Beijing\ 100871 ,$ $P.R.\ China $

\end{document}